\documentclass[a4paper,leqno,10pt]{article}

\usepackage{epsf,graphicx,epsfig}
\usepackage{amscd}
\usepackage{amsmath,latexsym,amssymb,amsthm}
\usepackage[nospace,noadjust]{cite}
\usepackage{textcomp}
\usepackage{setspace,cite}
\usepackage{lscape,fancyhdr,fancybox}
\usepackage{stmaryrd}
\usepackage[all,cmtip]{xy}
\usepackage{tikz}
\usepackage{cancel}
\usetikzlibrary{shapes,arrows,decorations.markings}
\usepackage{graphicx}

\usepackage{color}
\usepackage{url}
\usepackage{enumerate}
\usepackage[mathscr]{euscript}

  \theoremstyle{plain}

\swapnumbers
    \newtheorem{thm}{Theorem}[section]
    \newtheorem{prop}[thm]{Proposition}

    \newtheorem{subsec}[thm]{}
\theoremstyle{definition}
    \newtheorem{defn}[thm]{Definition}
        \newtheorem{remark}[thm]{Remark}
    \newtheorem{exam}[thm]{Example}

\theoremstyle{remark}

\title{}
\author{}
\date{}
\usepackage{amssymb}

\usepackage{hyperref}
\hypersetup{
	colorlinks,
	citecolor=blue,
	filecolor=black,
	linkcolor=blue,
	urlcolor=black
}

\begin{document}

\title{Cohomology and deformations of compatible Hom-Lie algebras}

\author{Apurba Das \footnote{Department of Mathematics,
Indian Institute of Technology, Kharagpur 721302, West Bengal, India.} \footnote{Email: apurbadas348@gmail.com}}



\maketitle
\begin{abstract}
In this paper, we consider compatible Hom-Lie algebras as a twisted version of compatible Lie algebras. Compatible Hom-Lie algebras are characterized as Maurer-Cartan elements in a suitable bidifferential graded Lie algebra. We also define a cohomology theory for compatible Hom-Lie algebras generalizing the recent work of Liu, Sheng and Bai. As applications of cohomology, we study abelian extensions and deformations of compatible Hom-Lie algebras.
\end{abstract}

\medskip

\medskip

\medskip

\medskip





\medskip

\medskip

\tableofcontents

\noindent
\thispagestyle{empty}


\section{Introduction}
The notion of Hom-algebras first appeared in the $q$-deformations of the Witt and Virasoro algebra by the work of Hartwig, Larsson and Silvestrov \cite{hls}. More precisely, they introduced Hom-Lie algebras as a twisted version of Lie algebras, where the usual Jacobi identity is twisted by a linear homomorphism (see Definition \ref{defn-hom-lie}). Later, various others algebras (e.g. associative, Leibniz, Poisson, \ldots) twisted by homomorphisms are also studied \cite{makh-sil,cheng-su,yau2}. These Hom-algebras are widely explored in the last 15 years. In \cite{amm-ej-makh,makh-sil2} the authors study cohomology and deformations of Hom-associative and Hom-Lie algebras. In particular, they generalize the classical Gerstenhaber bracket and Nijenhuis-Richardson bracket on the cochain complex of Hom-associative and Hom-Lie algebras. See also \cite{sheng-alg,yau2,yau3} and references therein for more on Hom-algebras.

\medskip

Two algebraic structures of the same kind are said to be compatible if their sum also defines a same type of algebraic structure. They appeared in many contexts of mathematics and mathematical physics. For instance, the notion of compatibility of two Poisson structures on a manifold was first appeared in the mathematical study of biHamiltonian mechanics \cite{kos,mag-mor}. Using the correspondence between Lie algebra structures on a vector space $\mathfrak{g}$ and linear Poisson structures on $\mathfrak{g}^\ast$, one lead to a notion of compatible Lie algebras \cite{kos}. See \cite{golu,pana} for more study on compatible Lie algebras. Some other compatible structures include compatible associative algebras \cite{odes}, compatible Lie bialgebras \cite{wu}, compatible Lie algebroids and Lie bialgebroids \cite{das-rep}. See also \cite{dotsenko,stro} for the operadic study of compatible algebraic structures.

\medskip

Recently, a cohomology theory for compatible Lie algebras has been introduced by Liu, Sheng and Bai \cite{sheng-comp}. This cohomology is based on the characterization of a compatible Lie algebra as Maurer-Cartan element in a bidifferential graded Lie algebra. It is also seen that this cohomology is related to extensions and deformations of compatible Lie algebras. These results have been extended to compatible associative algebras in \cite{chi-das-sami}. The present paper aims to define and study compatible Hom-Lie algebras. We observe that compatible Hom-Lie algebras are related to compatible Hom-Poisson manifolds. We first define representations and cohomology of a compatible Hom-Lie algebra. We relate our cohomology of a compatible Hom-Lie algebra with the cohomology of Hom-Lie algebra. As applications of our cohomology, we study abelian extensions and linear, finite order deformations of a compatible Hom-Lie algebra.

\medskip

The paper is organized as follows. In Section \ref{sec-2} (preliminary section), we recall Hom-Lie algebras and some basics on bidifferential graded Lie algebras. In Section \ref{sec-3}, we introduce compatible Hom-Lie algebras and their Maurer-Cartan characterizations in a suitably constructed bidifferential graded Lie algebra. We also define the notion of representation of a compatible Hom-Lie algebra and construct the semidirect product. The cohomology of a compatible Hom-Lie algebra with coefficients in a representation is given in Section \ref{sec-4}. We also introduce abelian extensions of a compatible Hom-Lie algebra and characterize equivalence classes of abelian extensions by the second cohomology group. In Section \ref{sec-5}, we first define linear deformations of a compatible Hom-Lie algebra and introduce Nijenhuis operators that induce trivial linear deformations. We show that the equivalence classes of infinitesimal deformations of a compatible Hom-Lie algebra are in one-to-one correspondence with the second cohomology group with coefficients in itself. Finally, we consider finite order deformations of compatible Hom-Lie algebra and study their extensions.

\medskip

All vector spaces, (multi)linear maps, tensor products are over a field $\mathbb{K}$ of characteristic zero.

\section{Hom-Lie algebras and bidifferential graded Lie algebras}\label{sec-2}
In this preliminary section, we recall some basics on Hom-Lie algebras  \cite{makh-sil,amm-ej-makh,sheng-alg} and bidifferential graded Lie algebras \cite{sheng-comp}.


\begin{defn}\label{defn-hom-lie}
A Hom-Lie algebra is a vector space $\mathfrak{g}$ together with a bilinear skew-symmetric bracket $[~,~] : \mathfrak{g} \otimes \mathfrak{g} \rightarrow \mathfrak{g}$ and a linear map $\alpha : \mathfrak{g} \rightarrow \mathfrak{g}$ satisfying the followings:
\begin{align*}
&\alpha [x, y] = [\alpha (x), \alpha (y) ] \qquad (\text{multiplicativity}), \\
&[[x, y], \alpha (z)] + [[y, z], \alpha (x)] + [[z, x], \alpha (y)] = 0 \qquad (\text{Hom-Jacobi identity}),
\end{align*}
for all $x,y, z \in \mathfrak{g}$.
\end{defn}

A Hom-Lie algebra as above may be denoted by the triple $(\mathfrak{g}, [~,~], \alpha)$ or simply by $\mathfrak{g}$ when no confusion arises. The bracket $[~,~]$ is said to be the Hom-Lie bracket on $\mathfrak{g}$ when the twisting map $\alpha$ is clear from the context.

It follows from the above definition that Hom-Lie algebras are a twisted version of Lie algebras. More specifically, a Hom-Lie algebra $(\mathfrak{g}, [~,~], \alpha)$ with $\alpha = \mathrm{id}$ is nothing but a Lie algebra. 

\begin{exam}
Let $(\mathfrak{g}, [~,~])$ be a Lie algebra and $\alpha : \mathfrak{g} \rightarrow \mathfrak{g}$ be a Lie algebra homomorphism. Then the triple $(\mathfrak{g}, [~,~]_\alpha = \alpha \circ [~,~], \alpha)$ is a Hom-Lie algebra, called induced by composition.
\end{exam}

\begin{exam}
Let $(A, \mu, \alpha)$ be a Hom-associative algebra, i.e., $A$ is a vector space, $\mu : A \otimes A \rightarrow A,~ (a, b) \mapsto a \cdot b$ is a bilinear map and $\alpha : A \rightarrow A$ a linear map satisfying $\alpha (a \cdot b) = \alpha (a) \cdot \alpha (b)$ and the following Hom-associativity
$(a \cdot b) \cdot \alpha (c) = \alpha(a) \cdot (b \cdot c),$ for $a, b, c \in A.$
If $(A, \mu, \alpha)$ is a Hom-associative algebra, then $(A, [~,~], \alpha)$ is a Hom-Lie algebra, where $[a, b] = a \cdot b - b \cdot a$, for $a, b \in A$.
\end{exam}

\begin{exam}
Let $(\mathfrak{g}, [~,~], \alpha)$ be a Hom-Lie algebra. Then for any $n \geq 0$, the triple $(\mathfrak{g}, [~,~]^{(n)} = \alpha^n \circ [~,~], \alpha^{n+1})$ is a Hom-Lie algebra, called the $n$-th derived Hom-Lie algebra.
\end{exam}

\begin{defn}
Let $(\mathfrak{g}, [~,~], \alpha)$ and $(\mathfrak{g}', [~,~]', \alpha')$ be two Hom-Lie algebras. A linear map $\phi : \mathfrak{g} \rightarrow \mathfrak{g}'$ is said to be a Hom-Lie algebra morphism if $\alpha' \circ \phi = \phi \circ \alpha$ and
$\phi [x,y] = [\phi (x), \phi (y)]',$ for $x, y \in \mathfrak{g}.$
\end{defn}

\medskip

Next we recall the graded Lie bracket (called the Nijenhuis-Richardson bracket) whose Maurer-Cartan elements are given by Hom-Lie algebra structures \cite{amm-ej-makh}. This generalizes the classical Nijenhuis-Richardson bracket in the context of Lie algebras \cite{nij-ric}. Let $\mathfrak{g}$ be a vector space and $\alpha : \mathfrak{g} \rightarrow \mathfrak{g}$ be a linear map. For each $n \geq 0$, consider the spaces $C^n_\mathrm{Hom} (\mathfrak{g} , \mathfrak{g})$ by
\begin{align*}
C^0_\mathrm{Hom} (\mathfrak{g}, \mathfrak{g}) = \{ x \in \mathfrak{g} |~ \alpha (x) = x \} ~~ \text{ and } ~~ C^n_\mathrm{Hom} (\mathfrak{g}, \mathfrak{g}) = \{ f : \wedge^n \mathfrak{g} \rightarrow \mathfrak{g} |~ \alpha \circ f = f \circ \alpha^{\wedge n} \}, ~ n \geq 1.
\end{align*} 
Then the shifted graded vector space $C^{\ast +1}_\mathrm{Hom} (\mathfrak{g}, \mathfrak{g}) = \oplus_{n \geq 0} C^{n+1}_\mathrm{Hom} (\mathfrak{g}, \mathfrak{g})$ carries a graded Lie bracket defined as follows. For $P \in C^{m+1}_\mathrm{Hom} (\mathfrak{g}, \mathfrak{g})$ and $Q \in C^{n+1}_\mathrm{Hom} (\mathfrak{g}, \mathfrak{g})$, the Nijenhuis-Richardson bracket $[P,Q]_\mathsf{NR} \in C^{m+n+1}_\mathrm{Hom} (\mathfrak{g}, \mathfrak{g})$ given by
\begin{align*}
&[P,Q]_\mathsf{NR} = P \diamond Q - (-1)^{mn}~ Q \diamond P, ~\text{ where }\\
&(P \diamond Q) (x_1, \ldots, x_{m+n+1}) = \sum_{\sigma \in Sh (n+1, m)} (-1)^\sigma ~ P \big(   Q ( x_{\sigma (1)}, \ldots, x_{\sigma (n+1)} ), \alpha^n (x_{\sigma (n+2)}), \ldots, \alpha^n (x_{\sigma (m+n+1)} \big).
\end{align*}

With this notation, we have the following.

\begin{prop}
Let $\mathfrak{g}$ be a vector space and $\alpha : \mathfrak{g} \rightarrow \mathfrak{g}$ be a linear map. Then the Hom-Lie brackets on $\mathfrak{g}$ are precisely the Maurer-Cartan elements in the graded Lie algebra $(C^{\ast + 1}_\mathrm{Hom} (\mathfrak{g}, \mathfrak{g}), [~,~]_\mathsf{NR} ).$
\end{prop}

\medskip

In the following, we recall the Chevalley-Eilenberg cohomology of a Hom-Lie algebra $(\mathfrak{g}, [~,~], \alpha)$ with coefficients in a representation.

\begin{defn}
Let $(\mathfrak{g}, [~,~], \alpha)$ be a Hom-Lie algebra. A representation of it consists of a vector space $V$ together with a bilinear operation (called the action) $\bullet : \mathfrak{g} \otimes V \rightarrow V$, $(x, v) \mapsto x \bullet v$ and a linear map $\beta : V \rightarrow V$ satisfying
\begin{align*}
&\beta (x \bullet v) = \alpha (x) \bullet \beta (v), \\
& [x,y] \bullet \beta (v) = \alpha (x) \bullet (y \bullet v) - \alpha (y) \bullet (x \bullet v),~ \text{ for } x, y \in \mathfrak{g}, v \in V.
\end{align*}
\end{defn}

We denote a representation as above by $(V, \bullet, \beta)$ or simply by $V$. It follows that any Hom-Lie algebra $(\mathfrak{g}, [~,~], \alpha)$ is a representation of itself with the action given by the bracket $[~,~].$ This is called the adjoint representation.

Let $(\mathfrak{g}, [~,~], \alpha)$ be a Hom-Lie algebra and $(V, \bullet, \beta)$ be a representation. For each $n \geq 0$, we define the $n$-th cochain group $C^n_\mathrm{Hom}(\mathfrak{g}, V)$ as
\begin{align*}
C^0_\mathrm{Hom} (\mathfrak{g}, V) = \{ v \in V |~ \beta (v) = v \} ~~ \text{ and } ~~ C^n_\mathrm{Hom} (\mathfrak{g}, V) = \{ f : \wedge^n \mathfrak{g} \rightarrow V |~ \beta \circ f = f \circ \alpha^{\wedge n} \}, ~ n \geq 1.
\end{align*}
The coboundary operator $\delta_\mathrm{Hom} : C^n_\mathrm{Hom} (\mathfrak{g}, V) \rightarrow C^{n+1}_\mathrm{Hom} (\mathfrak{g}, V)$, for $n \geq 0$, given by
\begin{align}
(\delta_\mathrm{Hom} v )(x) =~& x \bullet v, ~ \text{ for } v \in C^0_\mathrm{Hom} (\mathfrak{g}, V), x \in \mathfrak{g},\\
(\delta_\mathrm{Hom} f ) (x_1, \ldots, x_{n+1}) =& \sum_{i=1}^{n+1} (-1)^{i+1}~ \alpha^{n-1} (x_i) \bullet f (x_1, \ldots, \widehat{x_i}, \ldots, x_{n+1}) \\
&+ \sum_{1 \leq i < j \leq n+1} (-1)^{i+j}~ f ([x_i, x_j], \alpha (x_1), \ldots, \widehat{\alpha (x_i)}, \ldots, \widehat{\alpha (x_j)}, \ldots, \alpha (x_{n+1}) ), \nonumber
\end{align}
for $f \in C^n_\mathrm{Hom} (\mathfrak{g}, V)$ and $x_1, \ldots, x_{n+1} \in \mathfrak{g}$. The cohomology groups of the cochain complex $\{ C^\ast_\mathrm{Hom}(\mathfrak{g}, V), \delta_\mathrm{Hom} \}$ are called the Chevalley-Eilenberg cohomology groups, denoted by $H^\ast_\mathrm{Hom} (\mathfrak{g}, V).$

It is important to note that the coboundary operator for the Chevalley-Eilenberg cohomology of the Hom-Lie algebra $(\mathfrak{g}, [~,~], \alpha)$ with coefficients in itself is simply given by
\begin{align*}
\delta_\mathrm{Hom} f = (-1)^{n-1} [\mu, f]_\mathsf{NR}, ~\text{for } f \in C^n_\mathrm{Hom}(\mathfrak{g}, \mathfrak{g}),
\end{align*}
where $\mu \in C^2_\mathrm{Hom}(\mathfrak{g}, \mathfrak{g})$ corresponds to the Hom-Lie bracket $[~,~]$.

\medskip

\medskip

\noindent {\bf Bidifferential graded Lie algebras.}
Next, we recall bidifferential graded Lie algebras \cite{sheng-comp}.
Before that, let us first give the definition of a differential graded Lie algebra.

\begin{defn}
A differential graded Lie algebra is a triple $(L= \oplus L^i, [~,~], d)$ consisting of a graded Lie algebra together with a degree $+1$ differential $d : L \rightarrow L$ which is a derivation for the bracket $[~,~]$.
\end{defn}

An element $\theta \in L^1$ is said to be a Maurer-Cartan element in the differential graded Lie algebra $(L, [~,~], d)$ if $\theta$ satisfies
\begin{align*}
d \theta + \frac{1}{2} [\theta, \theta] = 0.
\end{align*}

\begin{defn}
A bidifferential graded Lie algebra is a quadruple $(L = \oplus L^i, [~, ~], d_1, d_2)$ in which the triples $(L, [~,~], d_1)$ and $(L, [~,~], d_2)$ are differential graded Lie algebras additionally satisfying
$d_1 \circ d_2 + d_2 \circ d_1 = 0.$
\end{defn}

\begin{remark}\label{dgla-bdgla}
Any graded Lie algebra can be considered as a bidifferential graded Lie algebra with both the differentials $d_1$ and $d_2$ to be trivial.
\end{remark}


\begin{defn}
Let $(L, [~,~], d_1, d_2)$ be a bidifferential graded Lie algebra. A pair of elements $(\theta_1, \theta_2) \in L^1 \oplus L^1$ is said to be a Maurer-Cartan element if
\begin{itemize}
\item[(i)] $\theta_1$ is a Maurer-Cartan element in the differential graded Lie algebra $(L, [~,~], d_1)$;
\item[(ii)] $\theta_2$ is a Maurer-Cartan element in the differential graded Lie algebra $(L, [~,~], d_2)$;
\item[(iii)] the following compatibility condition holds
\begin{align*}
d_1 \theta_2 + d_2 \theta_1 + [\theta_1, \theta_2] = 0.
\end{align*} 
\end{itemize}
\end{defn}

Like a differential graded Lie algebra can be twisted by a Maurer-Cartan element, the same result holds for bidifferential graded Lie algebras.

\begin{prop}\label{mc-deform} Let $(L, [~,~], d_1, d_2)$ be a bidifferential graded Lie algebra and let $(\theta_1, \theta_2)$ be a Maurer-Cartan element. Then the quadruple $(L, [~,~], d_1^{\theta_1}, d_2^{\theta_2})$ is a bidifferential graded Lie algebra, where
\begin{align*}
d_1^{\theta_1} = d_1 + [\theta_1, - ] ~~~ \text{ and } ~~~ d_2^{\theta_2} = d_2 + [\theta_2, - ].
\end{align*}
For any $\vartheta_1 , \vartheta_2 \in L^1$, the pair $(\theta_1 + \vartheta_1, \theta_2 + \vartheta_2)$ is a Maurer-Cartan element in the bidifferential graded Lie algebra $(L, [~,~], d_1, d_2)$ if and only if $(\vartheta_1, \vartheta_2)$ is a Maurer-Cartan element in the bidifferential graded Lie algebra $(L, [~,~], d_1^{\theta_1}, d_2^{\theta_2}).$
\end{prop}

\section{Compatible Hom-Lie algebras}\label{sec-3}

In this section, we introduce compatible Hom-Lie algebras and give a Maurer-Cartan characterization. We end this section by defining representations of compatible Hom-Lie algebras.

Let $\mathfrak{g}$ be a vector space and $\alpha : \mathfrak{g} \rightarrow \mathfrak{g}$ be a linear map.

\begin{defn}
Two Hom-Lie algebras $(\mathfrak{g}, [~,~]_1, \alpha)$ and $(\mathfrak{g}, [~,~]_2, \alpha)$ are said to be compatible if for all $\lambda, \eta \in \mathbb{K}$, the triple $(\mathfrak{g}, \lambda [~,~]_1 + \eta [~,~]_2, \alpha)$ is a Hom-Lie algebra.
\end{defn}

The condition in the above definition is equivalent to the following
\begin{align}\label{comp-cond-e}
[[x,y]_1, \alpha (z)]_2 + [[y,z]_1, \alpha(x)]_2 + [[z,x]_1, \alpha (y)]_2 + [[x,y]_2, \alpha (z)]_1 + [[y,z]_2, \alpha(x)]_1 + [[z,x]_2, \alpha (y)]_1 = 0,
\end{align}
for all $x,y, z \in \mathfrak{g}.$

\begin{defn}
A compatible Hom-Lie algebra is a quadruple $(\mathfrak{g}, [~,~]_1, [~,~]_2, \alpha)$ in which $(\mathfrak{g}, [~,~]_1, \alpha)$ and $(\mathfrak{g}, [~,~]_2, \alpha)$ are both Hom-Lie algebras and are compatible.
\end{defn}

In this case, we say that the pair $([~,~]_1, [~,~]_2)$ is a compatible Hom-Lie algebra structure on $\mathfrak{g}$ when the twisting map $\alpha$ is clear from the context.

Compatible Hom-Lie algebras are twisted version of compatible Lie algebras \cite{kos}. Recall that a compatible Lie algebra is a triple $(\mathfrak{g}, [~,~]_1, [~,~]_2)$ in which $(\mathfrak{g}, [~,~]_1)$ and $(\mathfrak{g}, [~,~]_2)$ are Lie algebras and are compatible in the sense that $\lambda [~,~]_1 + \eta [~,~]_2$ is a Lie bracket on $\mathfrak{g}$, for all $\lambda, \eta \in \mathbb{K}$. Thus, a compatible Hom-Lie algebra $(\mathfrak{g}, [~,~]_1, [~,~]_2, \alpha)$ with $\alpha = \mathrm{id}$ is nothing but a compatible Lie algebra.

\medskip

It is known that compatible Lie algebras are closely related with compatible Poisson structures \cite{kos}. Hence compatible Poisson structures can be thought of as a motivation to study compatible Lie algebras. In the following, we study compatible Hom-Poisson structures and relate them with compatible Hom-Lie algebras. Recall that a Hom-Poisson manifold is a triple $(M, \{~,~\}, \triangle)$ consisting of a smooth manifold $M$, a bilinear skew-symmetric operation $\{ ~,~ \} : C^\infty (M) \times C^\infty (M) \rightarrow C^\infty (M)$ and a linear map $\triangle : C^\infty (M) \rightarrow C^\infty (M)$ satisfying for $f, g, h \in C^\infty (M)$,
\begin{align*}
&\triangle \{ f, g \} = \{ \triangle f, \triangle g \}, \qquad \triangle (fg) = (\triangle f) (\triangle g),\\
&\{ f, gh \} = (\triangle g) \{ f, h \} + \{ f, g \} (\triangle h),\\
&\{ \{ f, g \}, \triangle h \} + \{ \{ g, h \}, \triangle f \} + \{ \{ h, f \}, \triangle g \} = 0.
\end{align*}
See \cite{cai} for more details. If $(M, \{~,~\})$ is a Poisson manifold and $\varphi : M \rightarrow M$ is a Poisson morphism then $(M, \{~,~\}_\varphi = \varphi^* \circ \{~,~\}, \triangle = \varphi^*)$ is a Hom-Poisson manifold.

Let $(M, \{~,~\}_1, \triangle)$ and $(M, \{~,~\}_1, \triangle)$ be two Hom-Poisson manifolds with same underlying manifold $M$ and same twisting map $\triangle$. These two Hom-Poisson manifolds are said to be compatible if for any $\lambda, \eta \in \mathbb{K}$, the triple $(M, \lambda \{ ~, ~\}_1 + \eta \{~, ~\}_2, \triangle)$ is also a Hom-Poisson manifold. In this case, we say that $(M, \{~,~\}_1, \{~,~\}_2, \triangle)$ is a compatible Hom-Poisson manifold. A compatible Hom-Poisson manifold $(M, \{~,~\}_1, \{~,~\}_2, \triangle)$ is said to be `linear' if $M$ is a vector space and $\{ ~, ~\}_1, \{~,~\}_2, \triangle$ takes linear maps to linear maps.

Let $(\mathfrak{g}, [~,~]_1, [~,~]_2, \alpha)$ be a compatible Hom-Lie algebra. For any $f \in C^\infty (\mathfrak{g}^*)$ and $\xi \in \mathfrak{g}^*$, we consider the tangent map of $f$ at the point $\alpha^* \xi$,
\begin{align*}
\mathfrak{g}^* ~ \cong ~ T_{\alpha^* \xi} \mathfrak{g}^* \xrightarrow{T_{\alpha^* \xi} f} T_{f (\alpha^* \xi)} \mathbb{K} ~ \cong ~\mathbb{K}.
\end{align*}
Since $T_{\alpha^* \xi} f$ is a linear map, it corresponds to an element of $\mathfrak{g}$. Using this notation, we will now define brackets $\{ ~, ~ \}_1, \{~, ~\}_1 : C^\infty (\mathfrak{g}^*) \times C^\infty (\mathfrak{g}^*) \rightarrow C^\infty (\mathfrak{g}^*)$ as follows:
\begin{align*}
\{ f, g \}_1 (\xi) := \langle [T_{\alpha^* \xi} f, T_{\alpha^* \xi} g]_1, \xi \rangle ~~~~ \text{ and } ~~~~ \{ f, g \}_2 (\xi) := \langle [T_{\alpha^* \xi} f, T_{\alpha^* \xi} g]_2, \xi \rangle, ~ \text{ for } f , g \in C^\infty (\mathfrak{g}^*).
\end{align*}
By using the fact that $(\mathfrak{g}, [~,~]_1, \alpha)$ is a Hom-Lie algebra, it is easy to verify that the triple $(\mathfrak{g}^*, \{~, ~\}_1, \triangle)$ is a Hom-Poisson manifold, where $\triangle : C^\infty (\mathfrak{g}^*) \rightarrow  C^\infty (\mathfrak{g}^*)$ is the map $\triangle (f) = f \circ \alpha^*$. Similarly, $(\mathfrak{g}, [~,~]_2, \alpha)$ is a Hom-Lie algebra implies that $(\mathfrak{g}^*, \{~, ~\}_2, \triangle)$ is a Hom-Poisson manifold. Finally, the compatibility condition of the compatible Hom-Lie algebra $(\mathfrak{g}, [~,~]_1, [~,~]_2, \alpha)$ ensures that $(\mathfrak{g}^*, \{~,~\}_1, \{~,~\}_2, \triangle)$ is a compatible Hom-Poisson manifold. Finally, if $f,g$ are two linear functions on $\mathfrak{g}^*$, then $\{ f, g \}_1, \{f, g \}_2$ and $\triangle (f)$ are all linear functions on $\mathfrak{g}^*$. Hence it follows that $(\mathfrak{g}^*, \{~,~\}_1, \{~, ~\}_2, \triangle)$ is a linear compatible Hom-Poisson manifold.

\begin{defn}
Let $(\mathfrak{g}, [~,~]_1, [~,~]_2, \alpha)$ and $(\mathfrak{g}', [~,~]'_1, [~,~]'_2, \alpha')$ be two compatible Hom-Lie algebras. A morphism between them is a linear map $\phi : \mathfrak{g} \rightarrow \mathfrak{g}'$ which is a Hom-Lie algebra morphism from $(\mathfrak{g}, [~,~]_1, \alpha)$ to $(\mathfrak{g}', [~,~]'_1, \alpha')$, and a Hom-Lie algebra morphism from $(\mathfrak{g}, [~,~]_2, \alpha)$ to $(\mathfrak{g}', [~,~]'_2, \alpha').$
\end{defn}

In the following, we give some examples of compatible Hom-Lie algebras.

\begin{exam}
Let $(\mathfrak{g}, [~,~]_1, [~,~]_2)$ be a compatible Lie algebra and $\alpha : \mathfrak{g} \rightarrow \mathfrak{g}$ be a compatible Hom-Lie algebra homomorphism, i.e., $\alpha$ is a Lie algebra homomorphism for both the Lie algebras $(\mathfrak{g}, [~,~]_1)$ and $(\mathfrak{g}, [~,~]_2)$. Then the quadruple $(\mathfrak{g}, \alpha \circ [~,~]_1, \alpha \circ [~,~]_2, \alpha)$ is a compatible Hom-Lie algebra.
\end{exam}

\begin{exam}
Let $(\mathfrak{g}, [~,~]_1, [~,~]_2, \alpha)$ be a compatible Hom-Lie algebra. Then for each $n \geq 0$, the quadruple $(\mathfrak{g}, [~,~]_1^{(n)} = \alpha^n \circ [~,~]_1, [~,~]_2^{(n)} = \alpha^n \circ [~,~]_2, \alpha^{n+1})$ is a compatible Hom-Lie algebra. This is the $n$-th derived compatible Hom-Lie algebra.
\end{exam}

Let $(\mathfrak{g}, [~,~], \alpha)$ be a Hom-Lie algebra. A Nijenhuis operator on this Hom-Lie algebra is a linear map $N: \mathfrak{g} \rightarrow \mathfrak{g}$ satisfying $\alpha \circ N = N \circ \alpha$ and
\begin{align*}
[Nx, Ny]= N ( [Nx, y] + [x, Ny] - N[x,y]),~ \text{ for } x, y \in \mathfrak{g}.
\end{align*}
Nijenhuis operators are useful to study linear deformations of a Hom-Lie algebra \cite{das-sen}. Consider the $4$-dimensional Hom-Lie algebra $(\mathfrak{g}, [~,~], \alpha)$, where $\mathfrak{g} = \langle e_1, e_2, e_3, e_4 \rangle$ and structure maps are given by
\begin{align*}
&[e_1, e_2 ] = a e_1 + a e_2, \\
&\alpha (e_1) = e_2, ~~~~ \alpha (e_2) = e_1, ~~~~ \alpha (e_3) = 0 ~~~ \text{ and } ~~~ \alpha (e_4 ) = e_3.
\end{align*}
Then it is easy to verify that the map $N : \mathfrak{g} \rightarrow \mathfrak{g}$ defined by
\begin{align*}
N (e_1) = e_2, ~~~~ N(e_2) = e_1, ~~~~ N (e_3) = e_3, ~~~~ N (e_4) = e_4
\end{align*}
is a Nijenhuis operator on the Hom-Lie algebra $(\mathfrak{g}, [~,~], \alpha)$.

\begin{exam}
Let $N$ be a Nijenhuis operator on a Hom-Lie algebra $(\mathfrak{g}, [~,~], \alpha)$. Then there is a deformed Hom-Lie bracket on $\mathfrak{g}$ given by
\begin{align*}
[x,y]_N := [Nx, y]+ [x, Ny] - N[x,y],~\text{ for } x, y \in \mathfrak{g}.
\end{align*}
In other words $(\mathfrak{g}, [~,~]_N, \alpha)$ is a Hom-Lie algebra. It is easy to see that the quadruple $(\mathfrak{g}, [~,~], [~,~]_N, \alpha)$ is a compatible Hom-Lie algebra.
\end{exam}

\begin{exam}
Let $(\mathfrak{g}, [~,~], \alpha)$ be a Hom-Lie algebra and $(V, \bullet, \beta)$ be a representation of it. Suppose $f \in C^2_\mathrm{Hom} (\mathfrak{g},V)$ is a $2$-cocycle in the Chevalley-Eilenberg cohomology complex of the Hom-Lie algebra $\mathfrak{g}$ with coefficients in $V$. Then the direct sum vector space $\mathfrak{g} \oplus V$ inherits a Hom-Lie algebra structure (called the $f$-twisted semidirect product) whose bracket is given by
\begin{align*}
[(x,u), (y, v)]_{\ltimes_f} := ([x,y], x \bullet v - y \bullet u + f (x,y)), ~ \text{ for } (x, u), (y, v) \in \mathfrak{g} \oplus V
\end{align*}
and the linear twisting map on $\mathfrak{g} \oplus V$ is given by $\alpha \oplus \beta$. It is easy to verify that the quadruple $(\mathfrak{g} \oplus V, [~,~]_{\ltimes_0}, [~,~]_{\ltimes_f}, \alpha \oplus \beta)$ is a compatible Hom-Lie algebra.
\end{exam}

Another example of a compatible Hom-Lie algebra arises from compatible Rota-Baxter operators on a Hom-Lie algebra. Rota-Baxter operators are an algebraic abstraction of the integral operator and operator analogue of Poisson structures \cite{guo-book}.

\begin{defn}
Let $(\mathfrak{g}, [~,~], \alpha)$ be a Hom-Lie algebra. A linear map $R : \mathfrak{g} \rightarrow \mathfrak{g}$ is said to be a Rota-Baxter operator of weight $\lambda \in \mathbb{K}$ on the Hom-Lie algebra $(\mathfrak{g}, [~,~], \alpha)$ if $R$ satisfies $\alpha \circ R = R \circ \alpha$ and
\begin{align*}
 [Rx, Ry] = R ([Rx,y]+ [x, Ry] + \lambda [x,y]),~ \text{ for } x, y \in \mathfrak{g}.
\end{align*}
\end{defn}

A Rota-Baxter operator $R$ induces a new Hom-Lie algebra structure on $\mathfrak{g}$ with the Hom-Lie bracket
\begin{align*}
[x,y]_R := [Rx,y]+ [x, Ry] + \lambda [x,y],~ \text{ for } x, y \in \mathfrak{g}.
\end{align*}
Compatible Poisson structures first appeared in the context of bihamiltonian mechanics \cite{kos,mag-mor}. The operator version of compatible Poisson structures in the Hom-Lie algebra context is given by the following.

\begin{defn}
Two Rota-Baxter operators $R$ and $S$ of same weight $\lambda \in \mathbb{K}$ on a Hom-Lie algebra $(\mathfrak{g}, [~,~], \alpha)$ are said to be compatible if
\begin{align*}
[Rx, Sy] + [Sx, Ry] = R([Sx, y]+[x, Sy]) + S ([Rx,y]+ [x, Ry]),~ \text{ for } x, y \in \mathfrak{g}.
\end{align*}
\end{defn}

\begin{exam}
Let $(\mathfrak{g}, [~,~], \alpha)$ be the Hom-Lie algebra given by $\mathfrak{g} = \langle e_1, e_2 \rangle$ and the structure maps are given by
\begin{align*}
[e_1, e_2 ] = a e_1 + a e_2, ~~~  \alpha (e_1 ) = e_2, ~~~ \alpha (e_2 ) = e_1.
\end{align*}
Then the map $R=\alpha$ is a Rota-Baxter operator of weight $\lambda = -1$.

Let $(\mathfrak{g}, [~,~], \alpha)$ be a Hom-Lie algebra and let $R: \mathfrak{g} \rightarrow \mathfrak{g}$ be a Rota-Baxter operator of weight $\lambda$. Then $- \lambda \mathrm{id} - R : \mathfrak{g} \rightarrow \mathfrak{g}$ is also a Rota-Baxter operator of the same weight. Moreover, $R$ and $- \lambda \mathrm{id} - R$ are compatible.
\end{exam}

The proof of the following proposition is straightforward.

\begin{prop}
Let $R$ and $S$ be two compatible Rota-Baxter operators of weight $\lambda \in \mathbb{K}$ on a Hom-Lie algebra $(\mathfrak{g},[~,~], \alpha).$ Then $(\mathfrak{g}, [~,~]_R, [~,~]_S, \alpha)$ is a compatible Hom-Lie algebra.
\end{prop}

\medskip

Let $\mathfrak{g}$ be a vector space and $\alpha : \mathfrak{g} \rightarrow \mathfrak{g}$ be a linear map. Consider the graded Lie algebra $(C^{\ast +1 }_\mathrm{Hom} (\mathfrak{g}, \mathfrak{g}), [~,~]_\mathsf{NR} )$ given in Section \ref{sec-2}. Hence by Remark \ref{dgla-bdgla}, the quadruple 
\begin{align*}
(C^{\ast +1 }_\mathrm{Hom} (\mathfrak{g}, \mathfrak{g}), [~,~]_\mathsf{NR} , d_1 = 0, d_2 = 0)
\end{align*}
is a bidifferential graded Lie algebra. Then we have the following Maurer-Cartan characterization of compatible Hom-Lie algebras.

\begin{thm}
There is a one-to-one correspondence between compatible Hom-Lie algebra structures on $\mathfrak{g}$ and Maurer-Cartan elements in the bidifferential graded Lie algebra $(C^{\ast +1 }_\mathrm{Hom} (\mathfrak{g}, \mathfrak{g}), [~,~]_\mathsf{NR} , d_1 = 0, d_2 = 0)$.
\end{thm}

\begin{proof}
Let $[~,~]_1$ and $[~,~]_2$ be two multiplicative skew-symmetric bilinear brackets on $\mathfrak{g}$. Then the brackets $[~,~]_1$ and $[~,~]_2$ correspond to elements (say, $\mu_1$ and $\mu_2$, respectively) in $C^2_\mathrm{Hom} (\mathfrak{g}, \mathfrak{g})$. Then
\begin{align*}
[~,~]_1 \text{ is a Hom-Lie bracket } \Leftrightarrow~& ~~ [\mu_1, \mu_1]_\mathsf{NR} = 0;\\
[~,~]_2 \text{ is a Hom-Lie bracket } \Leftrightarrow~& ~~ [\mu_2, \mu_2]_\mathsf{NR} = 0;\\
\text{ compatibility condition } (\ref{comp-cond-e}) ~\Leftrightarrow~& ~~  [\mu_1, \mu_2]_\mathsf{NR} = 0.
\end{align*}
Hence $(\mathfrak{g}, [~,~]_1, [~,~]_2, \alpha)$ is a compatible Hom-Lie algebra if and only if $(\mu_1, \mu_2)$ is a Maurer-Cartan element in the bidifferential graded Lie algebra $(C^{\ast +1 }_\mathrm{Hom} (\mathfrak{g}, \mathfrak{g}), [~,~]_\mathsf{NR} , d_1 = 0, d_2 = 0)$.
\end{proof}

Hence from Proposition \ref{mc-deform}, we get the following.

\begin{prop}
Let $(\mathfrak{g}, [~,~]_1, [~,~]_2, \alpha)$ be a compatible Hom-Lie algebra. Then for any multiplicative skew-symmetric bilinear operations $[~,~]_1'$ and $[~,~]_2'$ on $\mathfrak{g}$, the quadruple
\begin{align*}
(\mathfrak{g}, [~,~]_1 + [~,~]_1', [~,~]_2 + [~,~]_2', \alpha)
\end{align*}
is a compatible Hom-Lie algebra if and only if $(\mu_1', \mu_2')$ is a Maurer-Cartan element in the bidifferential graded Lie algebra $(C^{\ast +1}_\mathrm{Hom} (\mathfrak{g}, \mathfrak{g}), [~,~]_\mathsf{NR}, d_1 = [\mu_1, -], d_2 = [\mu_2, -])$. Here $\mu_1', \mu_2' \in C^2_\mathrm{Hom}(\mathfrak{g}, \mathfrak{g})$ denote the elements corresponding to the brackets $[~,~]_1'$ and $[~,~]_2'$, respectively.
\end{prop}

In the following, we define representations of a compatible Hom-Lie algebra.

\begin{defn}
A representation of the compatible Hom-Lie algebra $(\mathfrak{g}, [~,~]_1, [~,~]_2, \alpha)$ consists of a quadruple $(V, \bullet_1, \bullet_2, \beta)$ such that
\begin{itemize}
\item[(i)] $(V, \bullet_1, \beta)$ is a representation of the Hom-Lie algebra $(\mathfrak{g}, [~,~]_1, \alpha)$;
\item[(ii)] $(V, \bullet_2, \beta)$ is a representation of the Hom-Lie algebra $(\mathfrak{g}, [~,~]_2, \alpha)$;
\item[(iii)] the following compatibility condition holds
\begin{align*}
[x,y]_1 \bullet_2 \beta (v) + [x,y]_2 \bullet_1 \beta (v) = \alpha(x) \bullet_1 (y \bullet_2 v) - \alpha(y) \bullet_2 ( x \bullet_1 v) + \alpha(x) \bullet_2 (y \bullet_1 v) - \alpha(y) \bullet_1 ( x \bullet_2 v),
\end{align*}
for all $x,y \in \mathfrak{g}$ and $v \in V$.
\end{itemize}
\end{defn}

It follows that any compatible Hom-Lie algebra  $(\mathfrak{g}, [~,~]_1, [~,~]_2, \alpha)$  is a representation of itself, where $\bullet_1 = [~,~]_1$ and $\bullet_2 = [~,~]_2$. This is called the adjoint representation.

\begin{remark}\label{sum-rep}
Let $(\mathfrak{g}, [~,~]_1, [~,~]_2, \alpha)$ be a compatible Hom-Lie algebra and $(V, \bullet_1, \bullet_2, \beta)$ be a representation of it. Then for any $\lambda, \eta \in \mathbb{K}$, the triple $(\mathfrak{g}, \lambda [~,~]_1 + \eta [~,~]_2, \alpha)$ is a Hom-Lie algebra and $(V, \lambda ~\bullet_1 + \eta~ \bullet_2 , \beta)$ is a representation of it.
\end{remark}

The proof of the following proposition is similar to the standard case.

\begin{prop}\label{prop-comp-semi}
Let $(\mathfrak{g},[~,~]_1, [~,~]_2, \alpha)$ be a compatible Hom-Lie algebra and $(V, \bullet_1, \bullet_2, \beta)$ be a representation of it. Then the direct sum $\mathfrak{g} \oplus V$ carries a compatible Hom-Lie algebra structure with the linear homomorphism $\alpha \oplus \beta$, and Hom-Lie brackets
\begin{align*}
[(x,u), (y,v)]_i^\ltimes := ([x, y]_i, x \bullet_i v - y \bullet_i u),~ \text{ for } i=1,2 \text{ and } (x,u), (y, v) \in \mathfrak{g} \oplus V.
\end{align*}
This is called the semidirect product.
\end{prop}


\section{Cohomology of compatible Hom-Lie algebras}\label{sec-4}
In this section, we introduce the cohomology of a compatible Hom-Lie algebra with coefficients in a representation. We also define abelian extensions of a compatible Hom-Lie algebra and prove that isomorphism classes of abelian extensions are classified by the second cohomology group.

Let  $(\mathfrak{g}, [~,~]_1, [~,~]_2, \alpha)$  be a compatible Hom-Lie algebra and $(V, \bullet_1, \bullet_2, \beta)$ be a representation of it.  Let ${}^1\delta_\mathrm{Hom} : C^n_\mathrm{Hom}(\mathfrak{g}, V) \rightarrow C^{n+1}_\mathrm{Hom}(\mathfrak{g}, V)$  ~ (resp. ${}^2\delta_\mathrm{Hom} : C^n_\mathrm{Hom}(\mathfrak{g}, V) \rightarrow C^{n+1}_\mathrm{Hom}(\mathfrak{g}, V)$ ), for $n \geq 0$, be the coboundary operator for the Chevalley-Eilenberg cohomology of the Hom-Lie algebra $(\mathfrak{g}, [~,~]_1, \alpha)$ with coefficients in the representation $(V, \bullet_1, \beta)$  ~(resp. of the Hom-Lie algebra $(\mathfrak{g},[~,~]_2, \alpha)$ with coefficients in the representation $(V, \bullet_2, \beta)$ ). Then we have
\begin{align*}
( {}^1\delta_\mathrm{Hom})^2 = 0 ~~~~ \text{ and } ~~~~ ({}^2\delta_\mathrm{Hom})^2 = 0.
\end{align*}

Moreover, we have the following.

\begin{prop}\label{d-comp}
The coboundary operators ${}^1\delta_\mathrm{Hom}$ and ${}^2\delta_\mathrm{Hom}$ satisfy the following compatibility
\begin{align*}
{}^1\delta_\mathrm{Hom} \circ {}^2\delta_\mathrm{Hom} + {}^2\delta_\mathrm{Hom} \circ {}^1\delta_\mathrm{Hom} = 0.
\end{align*}
\end{prop}

Before we prove the above proposition, we first observe the followings. For a compatible Hom-Lie algebra $(\mathfrak{g}, [~,~]_1, [~,~]_2, \alpha)$ and a representation $(V, \bullet_1, \bullet_2, \beta)$, we consider the semidirect product compatible Hom-Lie algebra structure on $\mathfrak{g} \oplus V$ given in Proposition \ref{prop-comp-semi}. We denote by $\pi_1, \pi_2 \in C^2_\mathrm{Hom} (\mathfrak{g} \oplus V, \mathfrak{g} \oplus V)$ the elements corresponding to the Hom-Lie brackets $[~,~]_1^\ltimes$ and $[~,~]_2^\ltimes$ on $\mathfrak{g} \oplus V$, respectively. Let 
\begin{align*}
\delta^1_\mathrm{Hom} : C^n_\mathrm{Hom} (\mathfrak{g} \oplus V, \mathfrak{g} \oplus V) \rightarrow C^{n+1}_\mathrm{Hom} (\mathfrak{g} \oplus V, \mathfrak{g} \oplus V), ~~ \text{ for } n \geq 0,\\
\delta^2_\mathrm{Hom} : C^n_\mathrm{Hom} (\mathfrak{g} \oplus V, \mathfrak{g} \oplus V) \rightarrow C^{n+1}_\mathrm{Hom} (\mathfrak{g} \oplus V, \mathfrak{g} \oplus V), ~~ \text{ for } n \geq 0,
\end{align*}
denote respectively the coboundary operator for the Chevalley-Eilenberg cohomology of the Hom-Lie algebra $(\mathfrak{g} \oplus V, [~,~]_1^\ltimes, \alpha \oplus \beta)$ (resp. of the Hom-Lie algebra $(\mathfrak{g} \oplus V, [~,~]_2^\ltimes, \alpha \oplus \beta)$ ) with coefficients in itself. Note that any map $f \in C^n_\mathrm{Hom} (\mathfrak{g} \oplus V)$ can be lifted to a map $\widetilde{f} \in  C^n_\mathrm{Hom} (\mathfrak{g} \oplus V, \mathfrak{g} \oplus V) $ by
\begin{align*}
\widetilde{f} \big( (x_1, v_1), \ldots, (x_n, v_n) \big) = \big( 0, f (x_1, \ldots, x_n) \big).
\end{align*}
Then $f =0$ if and only if $\widetilde{f} =0$.

With these notations, for any $f \in C^n_\mathrm{Hom}(\mathfrak{g},V)$, we have
\begin{align*}
\widetilde{( {}^1\delta_\mathrm{Hom} f)}  = \delta^1_\mathrm{Hom} (\widetilde{f}) = (-1)^{n-1} [\pi_1, \widetilde{f}]_\mathsf{NR}  ~~~~ \text{ and } ~~~~ \widetilde{( {}^2\delta_\mathrm{Hom} f)}  = \delta^2_\mathrm{Hom} (\widetilde{f}) = (-1)^{n-1} [\pi_2, \widetilde{f}]_\mathsf{NR}.
\end{align*}

\medskip

\begin{proof} {\em (of Proposition \ref{d-comp})} For any $f \in C^n_\mathrm{Hom} (\mathfrak{g}, V)$, we have
\begin{align*}
&\widetilde{   ( {}^1\delta_\mathrm{Hom} \circ {}^2\delta_\mathrm{Hom} + {}^2\delta_\mathrm{Hom} \circ {}^1\delta_\mathrm{Hom})(f) } \\
&= \widetilde{{}^1\delta_\mathrm{Hom} ( {}^2\delta_\mathrm{Hom}{f})} ~ + ~ \widetilde{{}^2\delta_\mathrm{Hom} ( {}^1\delta_\mathrm{Hom} {f})} \\
&= (-1)^{n} ~  [\pi_1, \widetilde{{}^2\delta_\mathrm{Hom}{f}}    ]_\mathsf{NR} ~+~(-1)^{n} ~  [\pi_2, \widetilde{{}^1\delta_\mathrm{Hom}{f}}    ]_\mathsf{NR}     \\
&= - [\pi_1, [\pi_2, \widetilde{f}]_\mathsf{NR} ]_\mathsf{NR}  - [\pi_2, [\pi_1, \widetilde{f}]_\mathsf{NR} ]_\mathsf{NR}  \\
&= - [[\pi_1, \pi_2]_\mathsf{NR}, \widetilde{f}]_\mathsf{NR} ~+~ [\pi_2, [\pi_1, \widetilde{f}]_\mathsf{NR} ]_\mathsf{NR} - [\pi_2, [\pi_1, \widetilde{f}]_\mathsf{NR} ]_\mathsf{NR} \\
&= 0 ~~~~ \qquad (\because ~[\pi_1, \pi_2]_\mathsf{NR} = 0).
\end{align*}
Therefore, it follows that  $  ( {}^1\delta_\mathrm{Hom} \circ {}^2\delta_\mathrm{Hom} + {}^2\delta_\mathrm{Hom} \circ {}^1\delta_\mathrm{Hom})(f) = 0$. Hence the result follows.
\end{proof}

We are now in a position to define the cohomology of a compatible Hom-Lie algebra $(\mathfrak{g}, [~,~]_1, [~,~]_2, \alpha)$ with coefficients in a representation $(V, \bullet_1, \bullet_2, \beta)$. For each $n \geq 0$, we define an abelian group $C^n_\mathrm{cHom} (\mathfrak{g}, V)$ as follows:
\begin{align*}
C^0_\mathrm{cHom} (\mathfrak{g}, V) :=~& C^0_\mathrm{Hom} (\mathfrak{g}, V) \cap \{ v \in V | ~ x \bullet_1 v = x \bullet_2 v, \forall x \in \mathfrak{g} \} \\
=~& \{ v \in V |~ \beta (v) = v \text{ and } x \bullet_1 v = x \bullet_2 v, \forall x \in \mathfrak{g} \},
\end{align*}
\begin{align*}
C^n_\mathrm{cHom} (\mathfrak{g}, V) := \underbrace{C^n_\mathrm{Hom} (\mathfrak{g}, V) \oplus \cdots \oplus C^n_\mathrm{Hom} (\mathfrak{g}, V )}_{n \text{ copies }}, ~ \text{ for } n \geq 1.
\end{align*}

Define a map $\delta_{\mathrm{cHom}} : C^n_\mathrm{cHom} (\mathfrak{g}, V) \rightarrow C^{n+1}_\mathrm{cHom} (\mathfrak{g}, V)$, for $n \geq 0$ by
\begin{align*}
\delta_{\mathrm{cHom}}(v) (x) := x \bullet_1 v =  x \bullet_2 v, ~\text{ for } v \in C^0_{\mathrm{cHom}} (\mathfrak{g}, V) ~\text{ and } x \in \mathfrak{g}, \\
\delta_{\mathrm{cHom}} (f_1, \ldots, f_n ) := \big( {}^1\delta_\mathrm{Hom} f_1, \ldots, \underbrace{{}^1\delta_\mathrm{Hom} f_i + {}^2\delta_\mathrm{Hom} f_{i-1}}_{i\text{-th position}}, \ldots, {}^2\delta_\mathrm{Hom} f_n    \big),
\end{align*}
for $(f_1, \ldots, f_n) \in C^n_{\mathrm{cHom}} (\mathfrak{g}, V).$

Then we have the following.

\begin{prop}
The map $\delta_{\mathrm{cHom}}$ is a coboundary map, i.e., $(\delta_{\mathrm{cHom}})^2 = 0$.
\end{prop}

\begin{proof}
For any $v \in C^0_\mathrm{cHom} (\mathfrak{g}, V)$, we have
\begin{align*}
(\delta_\mathrm{cHom})^2 (v) = \delta_\mathrm{cHom} ( \delta_\mathrm{cHom} v ) =~& ({}^1\delta_\mathrm{Hom} \delta_\mathrm{cHom} v~,{}^2\delta_\mathrm{Hom} \delta_\mathrm{cHom} v ) \\
=~&  ( {}^1\delta_\mathrm{Hom} {}^1\delta_\mathrm{Hom} v~, {}^2\delta_\mathrm{Hom} {}^2\delta_\mathrm{Hom} v) = 0.
\end{align*}
Moreover, for any $(f_1, \ldots, f_n) \in C^n_\mathrm{cHom} (\mathfrak{g}, V)$, $n \geq 1$, we have
\begin{align*}
&(\delta_\mathrm{cHom})^2 (f_1, \ldots, f_n) \\
&= \delta_\mathrm{cHom}  \big(   {}^1 \delta_\mathrm{Hom} f_1, \ldots, {}^1 \delta_\mathrm{Hom} f_i + {}^2 \delta_\mathrm{Hom} f_{i-1}, \ldots, {}^2 \delta_\mathrm{Hom} f_n \big) \\
&= \big(   {}^1 \delta_\mathrm{Hom} {}^1 \delta_\mathrm{Hom} f_1~, {}^2 \delta_\mathrm{Hom} {}^1 \delta_\mathrm{Hom} f_1 + {}^1 \delta_\mathrm{Hom} {}^2 \delta_\mathrm{Hom} f_1 + {}^1 \delta_\mathrm{Hom} {}^1 \delta_\mathrm{Hom} f_2~, \ldots, \\
& \qquad \underbrace{  {}^2 \delta_\mathrm{Hom} {}^2 \delta_\mathrm{Hom}  f_{i-2} + {}^2 \delta_\mathrm{Hom} {}^1 \delta_\mathrm{Hom} f_{i-1} + {}^1 \delta_\mathrm{Hom} {}^2 \delta_\mathrm{Hom} f_{i-1} + {}^1 \delta_\mathrm{Hom} {}^1 \delta_\mathrm{Hom} f_i  }_{3 \leq i \leq n-1}, \ldots, \\
& \qquad {}^2 \delta_\mathrm{Hom} {}^2 \delta_\mathrm{Hom}  f_{n-1} + {}^2 \delta_\mathrm{Hom} {}^1 \delta_\mathrm{Hom} f_n + {}^1 \delta_\mathrm{Hom} {}^2 \delta_\mathrm{Hom} f_n ~, {}^2 \delta_\mathrm{Hom} {}^2 \delta_\mathrm{Hom} f_n \big) \\
&= 0 ~~~\quad (\text{from Proposition } \ref{d-comp}).
\end{align*}
This proves that $(\delta_\mathrm{cHom})^2 = 0$.
\end{proof}

It follows from the above proposition that $\{ C^\ast_{\mathrm{cHom}} (\mathfrak{g}, V), \delta_{\mathrm{cHom}} \}$ is a cochain complex. The corresponding cohomology groups
\begin{align*}
H^n_\mathrm{cHom} (\mathfrak{g}, V) := \frac{ Z^n_\mathrm{cHom} (\mathfrak{g}, V)}{B^n_\mathrm{cHom} (\mathfrak{g}, V)} = \frac{ \mathrm{Ker~ } \delta_\mathrm{cHom} : C^n_\mathrm{cHom} (\mathfrak{g}, V) \rightarrow  C^{n+1}_\mathrm{cHom} (\mathfrak{g}, V) }{ \mathrm{Im ~} \delta_\mathrm{cHom} : C^{n-1}_\mathrm{cHom} (\mathfrak{g}, V) \rightarrow  C^{n}_\mathrm{cHom} (\mathfrak{g}, V)
}, \text{ for } n \geq 0
\end{align*}
are called the cohomology of the compatible Hom-Lie algebra $(\mathfrak{g}, [~,~]_1, [~,~]_2, \alpha)$ with coefficients in the representation $(V, \bullet_1, \bullet_2, \beta)$.

It follows from the above definition that
\begin{align*}
H^0_\mathrm{cHom} (\mathfrak{g}, V) = \{ v \in V |~ \beta (v) = v \text{ and } x \bullet_1 v = x \bullet_2 v = 0, ~ \forall x \in \mathfrak{g} \}.
\end{align*}

A linear map $D : \mathfrak{g} \rightarrow V$ is said to be a derivation if $\beta \circ D = D \circ \alpha$ and
\begin{align*}
D [x,y]_1 = x \bullet_1 Dy - y \bullet_1 Dx, \qquad D[x,y]_2 = x \bullet_2 Dy - y \bullet_2 Dx, \text{ for } x, y \in \mathfrak{g}.
\end{align*}
We denote the space of derivations by $\mathrm{Der }(\mathfrak{g},V)$. A derivation $D$ is said to be inner if it is of the form $D = - \bullet_1 v = - \bullet_2 v$, for some $v \in C^0_\mathrm{cHom} (\mathfrak{g}, V)$. The space of inner derivations are denoted by $\mathrm{InnDer }(\mathfrak{g}, V)$. Then we have
\begin{align*}
H^1_\mathrm{cHom} (\mathfrak{g}, V) = \frac{ \mathrm{Der }(\mathfrak{g},V)  }{  \mathrm{InnDer }(\mathfrak{g},V) } = \mathrm{OutDer }(\mathfrak{g},V), \text{ the space of outer derivations.}
\end{align*}

\medskip

Let $\mathfrak{g} = (\mathfrak{g}, [~,~]_1, [~,~]_2, \alpha)$ be a compatible Hom-Lie algebra and $V = (V, \bullet_1, \bullet_2, \beta)$ be a representation of it. Then we know from Remark \ref{sum-rep} that $\mathfrak{g}_{+} = (\mathfrak{g}, [~,~]_1 + [~,~]_2 , \alpha)$ is a Hom-Lie algebra and $V_{+} = (V, \bullet_1 + \bullet_2 , \beta)$ is a representation of it. Consider the cochain complex $\{ C^\ast_\mathrm{cHom} (\mathfrak{g}, V), \delta_\mathrm{cHom} \}$ of the compatible Hom-Lie algebra $\mathfrak{g}$ with coefficients in the representation $V$, and the cochain complex $\{ C^\ast_\mathrm{Hom}(\mathfrak{g}_+, V_+), \delta_\mathrm{Hom} \}$ of the Hom-Lie algebra $\mathfrak{g}_+$ with coefficients in $V_+$.

For each $n \geq 0$, we define a map
\begin{align*}
\triangle_n : C^n_\mathrm{cHom} (\mathfrak{g}, V) \rightarrow C^n_\mathrm{Hom} (\mathfrak{g}_+, V_+) ~~~ \text{ by } ~~~ \begin{cases}  \triangle_0 (v) = \frac{1}{2} v, \\
\triangle_n ((f_1, \ldots, f_n )) = f_1 + \cdots + f_n, \end{cases}
\end{align*}
for $v \in C^0_\mathrm{cHom} (\mathfrak{g}, V)$ and $(f_1, \ldots, f_n) \in C^{n \geq 1}_\mathrm{cHom} (\mathfrak{g}, V)$. If $v \in C^0_\mathrm{cHom} (\mathfrak{g}, V)$, then 
\begin{align*}
(\delta_\mathrm{Hom} \circ \triangle_0 (v))(x) = \frac{1}{2} (\delta_\mathrm{Hom} (v))(x) = \frac{1}{2} (x \bullet_1 v + x \bullet_2 v) =~& x \bullet_1 v = x \bullet_2 v \\
=~& \delta_\mathrm{cHom}(v) (x) = (\triangle_1 \circ \delta_\mathrm{cHom} (v))(x).
\end{align*}
Moreover, if $(f_1, \ldots, f_n) \in C^{n \geq 1}_\mathrm{cHom}(\mathfrak{g}, V)$, then
\begin{align*}
(\delta_\mathrm{Hom} \circ \triangle_n) ((f_1, \ldots, f_n )) =~& \delta_\mathrm{Hom} (f_1 + \cdots + f_n) \\
=~& {}^1\delta_\mathrm{Hom} (f_1, \ldots, f_n ) + {}^2\delta_\mathrm{Hom} (f_1, \ldots, f_n ) \\
=~& \triangle_{n+1} \big(  {}^1\delta_\mathrm{Hom} f_1, \ldots, {}^1\delta_\mathrm{Hom} f_i + {}^2\delta_\mathrm{Hom} f_{i-1}, \ldots, {}^2\delta_\mathrm{Hom} f_n \big) \\
=~& (\triangle_{n+1} \circ \delta_\mathrm{cHom}) ((f_1, \ldots, f_n )).
\end{align*}
Therefore, we get the following.

\begin{thm}
The collection $\{ \triangle_n \}_{n \geq 0}$ defines a morphism of cochain complexes from $\{ C^\ast_\mathrm{cHom} (\mathfrak{g}, V), \delta_\mathrm{cHom} \}$ to $\{ C^\ast_\mathrm{Hom} (\mathfrak{g}_+, V_+), \delta_\mathrm{Hom} \}$. Hence, it induces a morphism $H^\ast_\mathrm{cHom}(\mathfrak{g},V) \rightarrow H^\ast_\mathrm{Hom}(\mathfrak{g}_+, V_+)$ between corresponding cohomologies.
\end{thm}

\medskip


\noindent {\bf Abelian extensions of compatible Hom-Lie algebras.} Here, we study abelian extensions of compatible Hom-Lie algebras and give a classification of equivalence classes of abelian extensions.

Let $\mathfrak{g} = (\mathfrak{g}, [~,~]_1, [~,~]_2, \alpha)$ be a compatible Hom-Lie algebra and $(V, \beta)$ be a pair of a vector space with a linear map. Note that $(V, \beta)$ can be considered as a compatible Hom-Lie algebra with both the Hom-Lie brackets on $V$ are trivial.

\begin{defn}
An abelian extension of $\mathfrak{g}$ by $V$ is an exact sequence of compatible Hom-Lie algebras
\begin{align}\label{ab-ext-ex}
\xymatrix{
0 \ar[r] &  (V, 0, 0, \beta) \ar[r]^{i} & (\mathfrak{h}, [~,~]_1^\mathfrak{h}, [~,~]_2^\mathfrak{h}, \alpha^\mathfrak{h}) \ar[r]^{j} & (\mathfrak{g}, [~,~]_1, [~,~]_2, \alpha) \ar[r] \ar@<+4pt>[l]^{s} & 0
}
\end{align}
together with a $\mathbb{K}$-splitting (given by $s$) satisfying
\begin{align}\label{alpha-comp}
\alpha^\mathfrak{h} \circ s = s \circ \alpha.
\end{align}
\end{defn}

We denote an abelian extension as above simply by $\mathfrak{h}$ when all the structures of the exact sequence (\ref{ab-ext-ex}) are understood.
Note that an abelian extension induces a representation of the compatible Hom-Lie algebra $(\mathfrak{g}, [~,~]_1, [~,~]_2, \alpha)$ on $(V, \beta)$ with the actions
\begin{align*}
x \bullet_1 v = [s(x), i(v)]_1^\mathfrak{h}   ~~~ \text{ and } ~~~
x \bullet_2 v = [s(x), i(v)]_2^\mathfrak{h}, \text{ for } x \in \mathfrak{g}, v \in V.
\end{align*}
It is easy to see that the above action is independent of the choice of $s$.

\begin{remark}
Let $(\mathfrak{h}, \alpha^\mathfrak{h})$ and $(\mathfrak{g}, \alpha)$ be two pairs of vector spaces endowed with linear maps. Suppose $j : \mathfrak{h} \rightarrow \mathfrak{g}$ is a linear map satisfying $\alpha \circ j = j \circ \alpha^\mathfrak{h}$. Then there might not be a linear map $s : \mathfrak{g} \rightarrow \mathfrak{h}$ that satisfies $\alpha^\mathfrak{h} \circ s = s \circ \alpha$. Take $\mathfrak{h} = \langle e_1, e_2 \rangle$ and $\alpha^\mathfrak{h} (e_1) = e_2$, $\alpha^\mathfrak{h}(e_2) = 0$. Also take $\mathfrak{g} = \langle f \rangle$ and $\alpha (f ) = 0$. Let $j : \mathfrak{h} \rightarrow \mathfrak{g}$ be the map defined by $j (e_1) = f$ and $j (e_2) = 0$. Let $s : \mathfrak{g} \rightarrow \mathfrak{h}$ be a map satisfying $\alpha^\mathfrak{h} \circ s = s \circ \alpha$. For $s (f) = \lambda e_1 + \eta e_2$, we have $f = (j \circ s) (f) = \lambda f$ which implies that $\lambda = 1$. Thus,
\begin{align*}
0 = (s \circ \alpha) (f) = (\alpha^\mathfrak{h} \circ s) (f) = \alpha^\mathfrak{h} ( e_1 + \eta e_2) = e_2
\end{align*}
leads to a contradiction.
\end{remark}

\begin{defn}
Two abelian extensions $\mathfrak{h}$ and $\mathfrak{h}'$ are said to be equivalent if there is a compatible Hom-Lie algebra homomorphism $\phi : \mathfrak{h} \rightarrow \mathfrak{h}'$ making the following diagram commutative
\[
\xymatrix{
0 \ar[r] &  (V, 0, 0, \beta)  \ar[r]^{i} \ar@{=}[d] & (\mathfrak{h}, [~,~]_1^\mathfrak{h}, [~,~]_2^\mathfrak{h}, \alpha^\mathfrak{h}) \ar[d]^{\phi} \ar[r]^{j} & (\mathfrak{g}, [~,~]_1, [~,~]_2, \alpha) \ar[r]  \ar@{=}[d] \ar@<+4pt>[l]^{s} & 0 \\
0 \ar[r] &  (V, 0, 0, \beta) \ar[r]_{i'} & (\mathfrak{h}', [~,~]_1^{\mathfrak{h}'}, [~,~]_2^{\mathfrak{h}'}, \alpha^{\mathfrak{h}'}) \ar[r]^{j'} & (\mathfrak{g}, [~,~]_1, [~,~]_2, \alpha) \ar[r] \ar@<+4pt>[l]^{s'} & 0.
}
\]
\end{defn}

Let $(\mathfrak{g}, [~,~]_1, [~,~]_2, \alpha)$ be a compatible Hom-Lie algebra and $(V, \bullet_1, \bullet_2, \beta)$ be a representation of it. We denote $\mathrm{Ext}(\mathfrak{g}, V)$ by the set of equivalence classes of abelian extensions of $\mathfrak{g}$ by $V$ for which the induced representation on $V$ is the given one.

With the above notation, we have the following.

\begin{thm}
There is a bijection $H^2_{\mathrm{cHom}} (\mathfrak{g}, V) \cong \mathrm{Ext}(\mathfrak{g}, V).$
\end{thm}

\begin{proof}
 Let $(f_1, f_2) \in Z^2_\mathrm{cHom} (\mathfrak{g}, V)$ be any $2$-cocycle. Then it induces a compatible Hom-Lie algebra structure on $\mathfrak{h} = \mathfrak{g} \oplus V$ with structure maps
 \begin{align*}
 [(x,u), (y, v)]_i^\ltimes :=~& ([x,y], x \bullet_i v -  y \bullet_i u + f_i (x, y)), ~ \text{ for } i = 1,2,\\
 \alpha^\mathfrak{h} (x, u) :=~& (\alpha (x), \beta (u)), ~ \text{ for } (x, u), (y, v) \in \mathfrak{h}.
 \end{align*}
 Moreover, this compatible Hom-Lie algebra makes
\begin{align*}
\xymatrix{
0 \ar[r] &  (V, 0, 0, \beta) \ar[r]^{i} & (\mathfrak{h}, [~,~]_1^\mathfrak{h}, [~,~]_2^\mathfrak{h}, \alpha^\mathfrak{h}) \ar[r]^{j} & (\mathfrak{g}, [~,~]_1, [~,~]_2, \alpha) \ar[r] \ar@<+4pt>[l]^{s} & 0
}
\end{align*} 
into an abelian extension with the obvious splitting $s$. It is easy to see that any other cohomologous $2$-cocycle gives rise to an equivalent abelian extension. Hence the map $H^2_\mathrm{cHom} (\mathfrak{g}, V) \rightarrow \mathrm{Ext}(\mathfrak{g}, V)$ is well defined.

 \medskip
 
 Conversely, let (\ref{ab-ext-ex}) be an abelian extension with splitting $s$. Then we may consider $\mathfrak{h} = \mathfrak{g} \oplus V$ and $s$ is the map $s(x) = (x, 0)$, for $x \in \mathfrak{g}.$ The map $i$ and $j$ are the obvious ones. Moreover, it follows from (\ref{alpha-comp}) that $\alpha^\mathfrak{h} = (\alpha, \beta)$. Finally, since $j$ is a compatible Hom-Lie algebra map, we have
 \begin{align*}
 j [(x,0), (y,0)]_i^\mathfrak{h} = [x, y]_i,~ \text{ for } x, y \in \mathfrak{g}.
\end{align*} 
This implies that  $[(x,0), (y,0)]_i^\mathfrak{h} = ([x,y]_i, f_i (x,y)),$ for $i=1,2$ and some $f_1, f_2 \in C^2_\mathrm{Hom}(\mathfrak{g}, V)$. As $([~,~]_1^\mathfrak{h}, [~,~]_2^\mathfrak{h})$ is a compatible Hom-Lie algebra structure on $\mathfrak{h}$, we get that the pair $(f_1, f_2) \in Z^2_\mathrm{cHom}(\mathfrak{g}, V)$ is a $2$-cocycle. It is left to the reader to check that equivalent abelian extensions give rise to cohomologous $2$-cocycles. Hence the map $\mathrm{Ext}(\mathfrak{g}, V) \rightarrow H^2_\mathrm{cHom} (\mathfrak{g}, V)$ is also well defined. Finally, these two maps are inverses to each other. Hence the proof. 
\end{proof}

\section{Deformations of compatible Hom-Lie algebras}\label{sec-5}
In this section, we study linear deformations and finite order deformations of a compatible Hom-Lie algebra generalizing the classical deformation theory of Gerstenhaber \cite{gers}. We introduce Nijenhuis operators that generate trivial linear deformations. We also define infinitesimal deformations of a compatible Hom-Lie algebra and characterize equivalence classes of infinitesimal deformations in terms of cohomology. Finally, we consider extensibility of finite order deformations.

Let $(\mathfrak{g}, [~,~]_1, [~,~]_2, \alpha)$ be a compatible Hom-Lie algebra and $(\omega_1, \omega_2) \in C^2_\mathrm{cHom} (\mathfrak{g}, \mathfrak{g})$ be a $2$-cochain. Define two bilinear operations on $\mathfrak{g}$ depending on the parameter $t$ as follows:
\begin{align*}
[x,y]_1^t := [x,y]_1 + t \omega_1 (x,y)  ~~ \text{ and } ~~
[x,y]_2^t := [x,y]_2 + t \omega_2 (x,y), \text{ for } x, y \in \mathfrak{g}.
\end{align*}

\begin{defn}
We say that $(\omega_1, \omega_2)$ generates a linear deformation of the compatible Hom-Lie algebra $(\mathfrak{g}, [~,~]_1, [~,~]_2, \alpha)$ if for all $t$, the quadruple $(\mathfrak{g}, [~,~]_1^t, [~,~]_2^t, \alpha)$ is a compatible Hom-Lie algebra.
\end{defn}

If $\mu_1, \mu_2 \in C^2_\mathrm{Hom}(\mathfrak{g}, \mathfrak{g})$ denote the elements corresponding to the Hom-Lie brackets $[~,~]_1$ and $[~,~]_2$, respectively, then the above definition is equivalent to saying that
\begin{align*}
[\mu_1 + t\omega_1, \mu_1 + t\omega_1]_\mathsf{NR} = 0, \qquad [\mu_2 + t \omega_2, \mu_2 + t \omega_2]_\mathsf{NR} = 0 ~~ \text{ and } ~~ [\mu_1 + t \omega_1, \mu_2 + t \omega_2]_\mathsf{NR} = 0.
\end{align*}
In other words, the followings are hold
\begin{align*}
[\mu_1, \omega_1]_\mathsf{NR} = 0,  \qquad    [\mu_2, \omega_2]_\mathsf{NR} = 0,  \qquad [\mu_1 , \omega_2]_\mathsf{NR} + [\mu_2, \omega_1]_\mathsf{NR} = 0, \\
[\omega_1, \omega_1]_\mathsf{NR} = 0,  \qquad    [\omega_2, \omega_2]_\mathsf{NR} = 0,  \qquad [\omega_1 , \omega_2]_\mathsf{NR} = 0.
\end{align*}
The first three condition implies that $(\omega_1, \omega_2) \in C^2_{\mathrm{cHom}} (\mathfrak{g}, \mathfrak{g})$ is a $2$-cocycle in the cohomology of the compatible Hom-Lie algebra $\mathfrak{g}$ with coefficients in the adjoint representation. Moreover, the last three conditions implies that $(\mathfrak{g}, \omega_1, \omega_2, \alpha)$ is a compatible Hom-Lie algebra.

\begin{defn}
Let $(\omega_1, \omega_2)$ and $(\omega_1', \omega_2')$ generate linear deformations $(\mathfrak{g}, [~,~]_1^t, [~,~]_2^t , \alpha)$ and $(\mathfrak{g}, [~,~]_1^{'t}, [~,~]_2^{'t} , \alpha)$ of a compatible Hom-Lie algebra $\mathfrak{g}.$ They are said to be equivalent if there exists a linear map $N: \mathfrak{g} \rightarrow \mathfrak{g}$ satisfying $\alpha \circ N = N \circ \alpha$ and such that 
\begin{align*}
\mathrm{id} + t N : (\mathfrak{g}, [~,~]_1^{t}, [~,~]_2^{t} , \alpha) \rightarrow (\mathfrak{g}, [~,~]_1^{'t}, [~,~]_2^{'t} , \alpha)
\end{align*}
is a morphism of compatible Hom-Lie algebras.
\end{defn}

One can equivalently write the explicit identities as follows:
\begin{align*}
\omega_i (x,y) - \omega_i' (x,y) =~& [x, Ny]_i + [Nx, y]_i - N[x,y]_i,\\
N\omega_i (x,y) =~& \omega_i' (x, Ny) + \omega_i' (Nx, y) + [Nx, Ny]_i,\\
\omega_i' (Nx, Ny)=~& 0,
\end{align*}
for $x, y \in \mathfrak{g}$ and $i=1,2$. Note that from the first identity, we get
\begin{align*}
(\omega_1, \omega_2) - (\omega_1', \omega_2') = \delta_\mathrm{cHom} N,
\end{align*}
where $N$ is considered as an element in $C^1_\mathrm{cHom} (\mathfrak{g}, \mathfrak{g})$. Hence, summarizing the above discussions, we get the following.

\begin{thm}\label{lin-def-thm}
Let $\mathfrak{g}$ be a compatible Hom-Lie algebra. Then there is a map from the set of equivalence classes of linear deformations of $\mathfrak{g}$ to the second cohomology group $H^2_\mathrm{cHom}(\mathfrak{g}, \mathfrak{g})$.
\end{thm}

\medskip

Next, we introduce trivial linear deformations of a compatible Hom-Lie algebra and introduce Nijenhuis operators that generate trivial linear deformations.

\begin{defn}
A linear deformation $([~,~]_1 + t \omega_1, [~,~]_2 + t \omega_2)$ of compatible Hom-Lie algebra $(\mathfrak{g}, [~,~]_1, [~,~]_2, \alpha)$ is said to be trivial if the deformation is equivalent to the undeformed one $([~,~]_1, [~,~]_2).$
\end{defn}

Thus, a linear deformation $([~,~]_1 + t \omega_1, [~,~]_2 + t \omega_2)$ is trivial if and only if there exists a linear map $N: \mathfrak{g} \rightarrow \mathfrak{g}$ satisfying $\alpha \circ N = N \circ \alpha$ and
\begin{align*}
\omega_i (x,y) =~& [x, Ny]_i + [Nx, y]_i - N[x,y]_i,\\
N\omega_i (x,y) =~& [Nx, Ny]_i,~ \text{ for } i=1,2 \text{ and } x,y \in \mathfrak{g}.
\end{align*}

This motivates us to introduce Nijenhuis operators on a compatible Hom-Lie algebra.

\begin{defn}
A Nijenhuis operator on a compatible Hom-Lie algebra $(\mathfrak{g}, [~,~]_1, [~,~]_2, \alpha)$ is a linear map $N:\mathfrak{g} \rightarrow \mathfrak{g}$ which is a Nijenhuis operator for both the Hom-Lie algebras $(\mathfrak{g},[~,~]_1, \alpha)$ and $(\mathfrak{g}, [~,~]_2, \alpha)$, i.e., $N$ satisfies $\alpha \circ N = N \circ \alpha$ and 
\begin{align*}
[Nx, Ny]_i = N ([Nx, y]_i + [ x, Ny]_i - N [x, y]_i),~ \text{ for } i=1,2 \text{ and } x, y \in \mathfrak{g}.
\end{align*}
\end{defn}

It follows that any trivial linear deformation of a compatible Hom-Lie algebra induces a Nijenhuis operator. The converse is given by the next result whose proof is straightforward.

\begin{prop}
Let $N : \mathfrak{g} \rightarrow \mathfrak{g}$ be a Nijenhuis operator on a compatible Hom-Lie algebra $(\mathfrak{g},[~,~]_1, [~,~]_2, \alpha)$. Then $(\omega_1, \omega_2)$ generates a trivial linear deformation of the compatible Hom-Lie algebra $(\mathfrak{g},[~,~]_1, [~,~]_2, \alpha)$, where
\begin{align*}
\omega_i (x, y) = [Nx, y]_i + [x, Ny]_i - N [x, y]_i,~ \text{ for } i=1,2 \text{ and } x, y \in \mathfrak{g}. 
\end{align*}
\end{prop}

In the following, we introduce infinitesimal deformations of a compatible Hom-Lie algebra as a generalization of linear deformations.

\begin{defn}
An infinitesimal deformation of a compatible Hom-Lie algebra $(\mathfrak{g},[~,~]_1, [~,~]_2, \alpha)$ is a linear deformation over $\mathbb{K}[[t]]/ (t^2)$.
\end{defn}

One can similarly define equivalences between two infinitesimal deformations. It is easy to see that any $2$-cocycle $(\omega_1, \omega_2) \in Z^2_\mathrm{cHom}(\mathfrak{g}, \mathfrak{g})$ induces an infinitesimal deformation and cohomologous $2$-cocycles give rise to equivalent infinitesimal deformations. Summarizing this fact with Theorem \ref{lin-def-thm}, we get the following.

\begin{thm}
Let $(\mathfrak{g},[~,~]_1, [~,~]_2, \alpha)$ be a compatible Hom-Lie algebra. Then the equivalence classes of infinitesimal deformations are in one-to-one correspondence with $H^2_\mathrm{cHom}(\mathfrak{g}, \mathfrak{g}).$
\end{thm}

\medskip

\noindent {\bf Finite order deformations.} Next, we consider finite order deformations of a compatible Hom-Lie algebra and study their extensions. Let $(\mathfrak{g}, [~,~]_1, [~,~]_2, \alpha)$ be a compatible Hom-Lie algebra. For any natural number $p \in \mathbb{N}$, consider the ring $\mathbb{K}[[t]]/(t^{p+1})$. Then $\mathfrak{g}[[t]]/(t^{p+1})$ is a module over $\mathbb{K}[[t]]/(t^{p+1})$. Note that the linear map $\alpha: \mathfrak{g} \rightarrow \mathfrak{g}$ extends to a $\mathbb{K}[[t]]/(t^{p+1})$-linear map (denoted by the same notation) on the space $\mathfrak{g}[[t]]/(t^{p+1})$.

\begin{defn}
An order $p$ deformation of the compatible Hom-Lie algebra $(\mathfrak{g}, [~,~]_1, [~,~]_2, \alpha)$ is a pair $(\mu_{1,t}, \mu_{2,t})$, where $\mu_{1,t} = \sum_{i=0}^p t^i \mu_{1,i}$ and $\mu_{2,t} = \sum_{i=0}^p t^i \mu_{2,i}$ with $\mu_{1,i}, \mu_{2,i} \in C^2_\mathrm{cHom}(\mathfrak{g}, \mathfrak{g})$, and $\mu_{1,0} = \mu_1$, $\mu_{2,0} = \mu_2$ which makes the quadruple $(\mathfrak{g}[[t]]/(t^{p+1}), \mu_{1,t}, \mu_{2,t}, \alpha)$ into a compatible Hom-Lie algebra over $\mathbb{K}[[t]]/(t^{p+1})$.
\end{defn}

Therefore, in an order $p$ deformation, we have
\begin{align*}
[\mu_{1,t}, \mu_{1,t} ]_\mathsf{NR} = 0, \quad [\mu_{2,t}, \mu_{2,t}]_\mathsf{NR} = 0 ~~~~~ \text{ and } ~~~~~ [\mu_{1,t}, \mu_{2,t}]_\mathsf{NR} = 0.
\end{align*}
These are equivalent to the following system of identities
\begin{align*}
{}^1 \delta_\mathrm{Hom}(\mu_{1,n}) =~& \frac{1}{2} \sum_{\substack{i+j = n \\ i, j \geq 1 }} [\mu_{1,i}, \mu_{1,j}]_\mathsf{NR}, \\
{}^2 \delta_\mathrm{Hom}(\mu_{2,n}) =~& \frac{1}{2} \sum_{\substack{i+j = n \\ i, j \geq 1 }} [\mu_{2,i}, \mu_{2,j}]_\mathsf{NR}, \\
{}^1 \delta_\mathrm{Hom} (\mu_{2,n}) + {}^2 \delta_\mathrm{Hom}(\mu_{1,n}) =~& \sum_{\substack{i+j = n \\ i, j \geq 1}} [\mu_{1,i}, \mu_{2,j} ]_\mathsf{NR},
\end{align*}
for $n=0,1, \ldots, p$.

\begin{defn}
An order $p$ deformation $(\mu_{1,t}, \mu_{2,t})$ of the compatible Hom-Lie algebra $(\mathfrak{g}, [~,~]_1, [~,~]_2, \alpha)$ is said to be extensible if there exists an element $(\mu_{1, p+1}, \mu_{2, p+1}) \in C^2_\mathrm{cHom}(\mathfrak{g}, \mathfrak{g})$ which makes the pair
\begin{align*}
\big( \overline{\mu_{1,t}} = \mu_{1,t} + t^{p+1} \mu_{1,p+1} , ~ \overline{\mu_{2,t}} = \mu_{2,t} + t^{p+1} \mu_{2,p+1}  \big)
\end{align*}
into a deformation of order $p+1$.
\end{defn}

Let $(\mu_{1,t}, \mu_{2,t})$ be an order $p$ deformation of the Hom-Lie algebra $(\mathfrak{g}, [~,~]_1, [~,~]_2, \alpha)$. Define an element $Ob_{ (\mu_{1,t}, \mu_{2,t}) } \in C^3_\mathrm{cHom} (\mathfrak{g}, \mathfrak{g})$ by
\begin{align*}
Ob_{ (\mu_{1,t}, \mu_{2,t}) } = \big( \frac{1}{2} \sum_{\substack{i+j = p+1 \\ i, j \geq 1 }} [\mu_{1,i}, \mu_{1,j}]_\mathsf{NR}, ~ \sum_{\substack{i+j = p+1 \\ i, j \geq 1}} [\mu_{1,i}, \mu_{2,j} ]_\mathsf{NR}, ~ \frac{1}{2} \sum_{\substack{i+j = p+1 \\ i, j \geq 1 }} [\mu_{2,i}, \mu_{2,j}]_\mathsf{NR}  \big). 
\end{align*}
Since $\frac{1}{2} \sum_{\substack{i+j = p+1 \\ i, j \geq 1 }} [\mu_{1,i}, \mu_{1,j}]_\mathsf{NR}$ is the obstruction cochain to extend the order $p$ deformation of the Hom-Lie algebra $(\mathfrak{g}, [~,~]_1, \alpha)$, we have ${}^1 \delta_\mathrm{Hom} (\frac{1}{2} \sum_{\substack{i+j = p+1 \\ i, j \geq 1 }} [\mu_{1,i}, \mu_{1,j}]_\mathsf{NR}) = 0$. Similarly, ${}^2 \delta_\mathrm{Hom} (\frac{1}{2} \sum_{\substack{i+j = p+1 \\ i, j \geq 1 }} [\mu_{2,i}, \mu_{2,j}]_\mathsf{NR}) = 0$. Moreover, we have
\begin{align*}
&{}^1 \delta_\mathrm{Hom} \big(  \sum_{\substack{i+j = p+1 \\ i, j \geq 1}} [\mu_{1,i}, \mu_{2,j} ]_\mathsf{NR}  \big) ~+~ {}^2 \delta_\mathrm{Hom}  \big(   \frac{1}{2} \sum_{\substack{i+j = p+1 \\ i, j \geq 1 }} [\mu_{1,i}, \mu_{1,j}]_\mathsf{NR} \big) \\
&= \sum_{\substack{i+j = p+1 \\ i, j \geq 1}} [\mu_1, [\mu_{1,i}, \mu_{2,j} ]_\mathsf{NR} ]_\mathsf{NR} ~+~ \frac{1}{2} \sum_{\substack{i+j = p+1 \\ i, j \geq 1 }} [\mu_2, [\mu_{1,i}, \mu_{1,j}]_\mathsf{NR}  ]_\mathsf{NR} \\
&= \sum_{\substack{i+j = p+1 \\ i, j \geq 1}} \big(  [ [\mu_1, \mu_{1,i}]_\mathsf{NR}, \mu_{2,j}]_\mathsf{NR} - [\mu_{1,i}, [\mu_1, \mu_{2,j}]_\mathsf{NR} ]_\mathsf{NR} \big) \\
& \quad + \frac{1}{2} \sum_{\substack{i+j = p+1 \\ i, j \geq 1}} \big(    [[\mu_2, \mu_{1,i}]_\mathsf{NR}, \mu_{1,j}]_\mathsf{NR} - [\mu_{1,i},  [\mu_2, \mu_{1,j}]_\mathsf{NR}]_\mathsf{NR} \big) \\
&= \sum_{\substack{i+j = p+1 \\ i, j \geq 1}}  [ [\mu_1, \mu_{1,i}]_\mathsf{NR}, \mu_{2,j}]_\mathsf{NR} - \sum_{\substack{ i+j = p+1 \\ i,j \geq 1}} [\mu_{1,i}, [\mu_1, \mu_{2,j}]_\mathsf{NR} + [\mu_2, \mu_{1,j}]_\mathsf{NR} ]_\mathsf{NR} \\
&= - \sum_{\substack{ i_1 + i_2 + j = p+1 \\ i_1, i_2, j \geq 1}} \frac{1}{2} [[\mu_{1, i_1}, \mu_{1, i_2}]_\mathsf{NR}, \mu_{2,j}]_\mathsf{NR} + \sum_{\substack{i+ j_1 + j_2 = p+1 \\ i, j_1, j_2 \geq 1}} [\mu_{1,i}, [\mu_{1, j_1}, \mu_{2,j_2}]_\mathsf{NR} ]_\mathsf{NR} \\
&= - \sum_{\substack{ i  + j +k = p+1 \\ i, j, k \geq 1}} \frac{1}{2} [[\mu_{1, i}, \mu_{1, j}]_\mathsf{NR}, \mu_{2,k}]_\mathsf{NR} + \sum_{\substack{ i  + j +k = p+1 \\ i, j, k \geq 1}} \frac{1}{2} [[\mu_{1, i}, \mu_{1, j}]_\mathsf{NR}, \mu_{2,k}]_\mathsf{NR} = 0.
\end{align*}
Similarly, we can show that
\begin{align*}
{}^1 \delta_\mathrm{Hom} \big( \frac{1}{2} \sum_{\substack{i+j = p+1 \\ i, j \geq 1 }} [\mu_{2,i}, \mu_{2,j}]_\mathsf{NR}  \big) ~+~ {}^2 \delta_\mathrm{Hom} \big( \sum_{\substack{i+j = p+1 \\ i, j \geq 1}} [\mu_{1,i}, \mu_{2,j} ]_\mathsf{NR}  \big) = 0.
\end{align*}
Therefore, we have $\delta_\mathrm{cHom} ( Ob_{ (\mu_{1,t}, \mu_{2,t}) } ) = 0$. The corresponding cohomology class $[Ob_{ (\mu_{1,t}, \mu_{2,t}) } ] \in H^3_\mathrm{cHom}(\mathfrak{g}, \mathfrak{g})$ is called the obstruction class to extend the deformation $(\mu_{1,t}, \mu_{2,t}).$

\begin{thm}
An order $p$ deformation $(\mu_{1,p}, \mu_{2,p})$ of the compatible Hom-Lie algebra $(\mathfrak{g}, [~,~]_1, [~,~]_2, \alpha)$ is extensible if and only if the corresponding obstruction class $[Ob_{ (\mu_{1,t}, \mu_{2,t}) } ] \in H^3_\mathrm{cHom}(\mathfrak{g}, \mathfrak{g})$ is trivial.
\end{thm}

\begin{proof}
Suppose $(\mu_{1,t}, \mu_{2,t})$ is extensible. Let $(\mu_{1, p+1}, \mu_{2,p+1}) \in C^2_\mathrm{cHom}(\mathfrak{g}, \mathfrak{g})$ be an element which makes the pair $\big( \overline{\mu_{1,t}} = \mu_{1,t} + t^{p+1} \mu_{1,p+1} , ~ \overline{\mu_{2,t}} = \mu_{2,t} + t^{p+1} \mu_{2,p+1}  \big)$ into a deformation of order $p+1$. Then it follows from the expression of $Ob_{ (\mu_{1,t}, \mu_{2,t}) } $ that
\begin{align*}
Ob_{ (\mu_{1,t}, \mu_{2,t}) }  = \delta_\mathrm{cHom} \big( (\mu_{1, p+1}, \mu_{2,p+1}) \big).
\end{align*}
This shows that the cohomology class $[Ob_{ (\mu_{1,t}, \mu_{2,t}) } ]$ is trivial.

Conversely, suppose $(\mu_{1,t}, \mu_{2,t})$ is an order $p$ deformation for which $[Ob_{ (\mu_{1,t}, \mu_{2,t}) } ] \in H^3_\mathrm{cHom}(\mathfrak{g}, \mathfrak{g})$ is trivial. Then we have $ Ob_{ (\mu_{1,t}, \mu_{2,t}) }  = \delta_\mathrm{cHom} \big( (\mu_{1, p+1}, \mu_{2, p+1}) \big) $, for some $(\mu_{1, p+1}, \mu_{2, p+1}) \in C^2_\mathrm{cHom} (\mathfrak{g}, \mathfrak{g})$. Then it is easy to see that   $\big( \overline{\mu_{1,t}} = \mu_{1,t} + t^{p+1} \mu_{1,p+1} , ~ \overline{\mu_{2,t}} = \mu_{2,t} + t^{p+1} \mu_{2,p+1}  \big)$ is a deformation of order $p+1$. In other words, $(\mu_{1,t}, \mu_{2,t})$ is extensible.
\end{proof}

\medskip
\noindent {\bf Acknowledgements.} The author would like to thank Indian Institute of Technology (IIT) Kharagpur for providing the beautiful academic atmosphere where the research has been carried out.
\medskip


%

\end{document}